\documentclass[12pt,leqno]{amsart}
\usepackage{graphicx}
\usepackage[margin=1.3in]{geometry}
\usepackage{amsmath}
\usepackage{amssymb}
\usepackage{amsthm}
\usepackage{amsfonts}
\usepackage{mathrsfs}
\usepackage{enumerate}
\usepackage{enumitem}
\usepackage[scriptsize,hang,raggedright]{subfigure}
\usepackage[cmyk]{xcolor}
\usepackage{mathtools}

\usepackage{hyperref}
\usepackage{pdfsync}
\usepackage{dsfont}
\usepackage{color}

\usepackage[normalem]{ulem}

\usepackage{cite}
\usepackage{bbm}
\usepackage{multirow}

\allowdisplaybreaks

\newtheorem{theorem}{Theorem}
\newtheorem{proposition}[theorem]{Proposition}
\newtheorem{lemma}[theorem]{Lemma}
\newtheorem{corollary}[theorem]{Corollary}

\theoremstyle{remark}
\newtheorem{remark}[theorem]{Remark}
\theoremstyle{definition}
\newtheorem{definition}[theorem]{Definition}
\newtheorem{example}[theorem]{Example}
\newcommand{\R}{\mathbb{R}}

\newcommand{\ba}{\begin{array}}
\newcommand{\ea}{\end{array}}
\newcommand{\trns}[1]{\widehat{#1}\,}

\newcommand{\bthm}{\begin{theorem}}
\newcommand{\ethm}{\end{theorem}}
\newcommand{\bprop}{\begin{proposition}}
\newcommand{\eprop}{\end{proposition}}
\newcommand{\blemma}{\begin{lemma}}
\newcommand{\elemma}{\end{lemma}}
\newcommand{\bexmpl}{\begin{example}}
\newcommand{\eexmpl}{\end{example}}

\newcommand{\beqn}{\begin{equation}}
\newcommand{\eeqn}{\end{equation}}
\newcommand{\beqns}{\begin{equation*}}
\newcommand{\eeqns}{\end{equation*}}

\newcommand{\pt}{\partial}

\newcommand{\Rd}{\mathbb{R}^d}

\newcommand{\ol}{\overline}

\newcommand{\Hd}{\mathcal{H}^{d-1}}

\renewcommand{\leq}{\leqslant}
\renewcommand{\geq}{\geqslant}

\definecolor{mygreen}{rgb}{0.1,0.75,0.2}

\newcommand{\N}{\mathbb{N}}

\newcommand{\E}{\mathbf{E}}

\newcommand{\Om}{\Omega}

\newcommand{\FF}{\mathbf{F}}
\newcommand{\GG}{\mathbf{G}}

\newcommand{\calV}{\mathcal{V}}
\newcommand{\calF}{\mathcal{F}}

\newcommand{\FNm}{\mathbf{F}_{N,\underline m}}

\DeclareMathOperator{\Per}{Per}
\DeclareMathOperator{\dive}{div}

\DeclareMathOperator{\loc}{loc}
\title[Droplet breakup in the liquid drop model with background potential]{Droplet breakup in the liquid drop model with background potential}

\author{Stan Alama}
\address{Department of Mathematics and Statistics,
				McMaster University, Hamilton, ON}
\email{alama@mcmaster.ca}
\author{Lia Bronsard}
\address{Department of Mathematics and Statistics,
				McMaster University, Hamilton, ON}
\email{bronsard@mcmaster.ca}
\author{Rustum Choksi}
\address{Department of Mathematics and Statistics,
				McGill University, Montr\'{e}al, QC}
\email{rustum.choksi@mcgill.ca}
\author{Ihsan Topaloglu}
\address{Department of Mathematics and Applied Mathematics,
				Virginia Commonwealth University, Richmond, VA}
\email{iatopaloglu@vcu.edu}

\date{\today}                                        
\subjclass{35Q40, 35Q70, 49Q20, 49S05, 82D10}
\keywords{liquid drop model, droplet breakup, background potential, nonlocal isoperimetric problem,  generalized minimizer, concentration-compactness method}

\begin{document}

\begin{abstract}  We consider a variant of Gamow's liquid drop model, with a general repulsive Riesz kernel and a long-range attractive background potential with weight $Z$.  The addition of the background potential acts as a regularization for the liquid drop model in that it restores the existence of minimizers for arbitrary mass. 
We consider the regime of small $Z$ and characterize the structure of minimizers in the limit $Z\to 0$ by means of a sharp asymptotic expansion of the energy.  In the process of studying this limit we characterize all minimizing sequences for the Gamow model in terms of ``generalized minimizers''.
\end{abstract}

\maketitle

\baselineskip=16pt

%%%%%%%%%%%%%%%%%%%%%%%%%%%%%%%%%%%%%%%%
%%%%%%%%%%%%%%%%%%%%%%%%%%%%%%%%%%%%%%%%
%%%%%%%%%%%%%%%%%%%%%%%%%%%%%%%%%%%%%%%%
\section{Introduction}\label{sec:intro}
%%%%%%%%%%%%%%%%%%%%%%%%%%%%%%%%%%%%%%%%
%%%%%%%%%%%%%%%%%%%%%%%%%%%%%%%%%%%%%%%%
%%%%%%%%%%%%%%%%%%%%%%%%%%%%%%%%%%%%%%%%

We consider the  following variational problem:  
	\beqn \label{eqn:e_Z}
		e_Z(M) \,:=\, \inf \left\{ \E_Z(\Om) \colon  \Om \subset \Rd, \ |\Om|=M\right\},
	\eeqn
where the energy functional $\E_Z$ is defined as
	\beqn \label{eqn:energy}
		\E_Z(\Om) :=  \Per(\Om) + \int_{\Om}\!\int_{\Om} \frac{dx\,dy}{|x-y|^s} - Z\int_{\Om} \frac{dx}{|x|^p}
	\eeqn
with $0<p<s<d$ and $d \geq 2$. Here the first term is the perimeter of the set $\Om$ in the sense of Caccioppoli and is given by 
	\[ 
	\Per(\Om)=\sup \left\{
 \int_{\Om} \dive\phi \,dx \colon \phi \in C_{0}^1(\Rd;\Rd), \ \|\phi\|_{L^\infty(\Rd)} \leq 1\right\}.
	\]

\medskip

Our main motivation for this problem and the consideration of the small $Z$ regime stems from Gamow's {\em liquid drop model} \cite{Ga1930} which successfully models the shape of an atomic nucleus. 
Gamow's model is essentially equivalent to the minimization problem \eqref{eqn:e_Z} with 
 $d=3$,  Coulombic repulsion $s=1$, and $Z=0$:  
\beqn \label{eqn:LD}
   {\rm minimize} \quad \Per(\Om) + \int_{\Om}\!\int_{\Om} \frac{dx\,dy}{|x-y|} \quad {\rm over} \quad \Om\subset\R^3 \quad \text{ with } \quad |\Omega| = M.
   \eeqn
 This problem recently 
 resurfaced in the context of the Ohta--Kawasaki model for self-assembly of diblock copolymers (cf. \cite{ChPe2010,ChPe2011}), and has since attracted much mathematical interest (cf. \cite{BoCr14,FrKiNa2016,FrLi2015,KnMu2013,KnMu2014,KnMuNo2016,LO1,RW2014,RW2017} as well as \cite{ChMuTo2017} for a general overview). 
One of the fundamental characteristics of the liquid drop model is that it  predicts the spherical shape
of small nuclei and the non-existence of arbitrarily large nuclei.  It is
precisely the competition between opposing forces (the
surface tension and Coulombic repulsion) 
 which makes proving these predictions non-trivial. 
 The non-existence of minimizers for large $M$ is associated with the breakup of droplets tending to infinity.

From a physical point of view, though, one might expect other forces to be present which restore existence for larger values of $M$, predicting a structured configuration  of droplets.  
One way to introduce such effects is to introduce an attractive ``background 
nucleus", which is effected by adding to \eqref{eqn:LD} an external attractive potential of the form 
 \beqn \label{eq-ext-pot}  
 	V(x) = -  \frac{Z}{|x|^p},
 \eeqn
for $Z>0$ and $0<p \leq 1$. 
Here we take the ``background nucleus" to be centered at the origin, and of longer range, in the sense that they have slower decay than the Coulombic nonlocal interaction term.
The physical case of $p=1$ (Coulombic attraction) was recently considered 
 by  Lu and Otto  \cite{LO2}, and by Frank, Nam and van den Bosch \cite{FrNaBo2016} where it was proved that the effect  of $V$ simply increases the critical threshold in $M$ for the non-existence of minimizers.  On the other hand, choosing a potential with $p<1$ restores existence for all $M$ (cf. Theorem \ref{thm:anyZ} and \cite{AlBrChTo2017_1}); we may think of the addition of the attractive long-range potential as {\it regularizing} the generalized liquid drop model \eqref{eqn:energy}.  We then focus on the  structure of minimizers in the small $Z$ regime. In doing so, we completely describe particular configurations of {\it generalized minimizers} (cf. \cite{KnMuNo2016}, Definition \ref{def:genmin}) of the liquid drop model.

\medskip

Our first result confirms that the presence of the external  potential  \eqref{eq-ext-pot} with  $p<s$ indeed restores existence for all masses $M>0$.  
\begin{theorem}\label{thm:anyZ}
For all $Z>0$ and for any $M>0$, the minimum $e_Z(M)$ is attained.
\end{theorem}
%It is in this sense that we speak of $\E_Z$ as a {\it regularization} of $\E_0$, the energy functional with $Z=0$.
This result is a generalization of the result in 
\cite{AlBrChTo2017_1}, and for convenience we will present an outline of the proof (which differs from that in \cite{AlBrChTo2017_1})  in section 2.
Our principal interest is in studying minimizers of $\E_Z$ in the limit $Z\to 0$.  For $d \geq 2$, it is well-known that there exists $m^*=m^*(d,s)>0$ such that the $Z=0$ problem,
	\beqn \label{eqn:e_0}
		e_0(M) \,:=\, \inf \left\{ \E_0(\Om) \colon  \Om \subset \Rd, \ |\Om|=M\right\}
	\eeqn
does not admit minimizers for $M>m^*$ and $s\in(0,2)$ (see \cite[Theorem 2.5]{KnMu2013} and \cite[Theorem 3.3]{KnMu2014}, and also \cite{LO1} and \cite{FrKiNa2016} for the case $d=3$, $s=1$).  Thus, when $M>m^*$ a sequence of minimizers $\Om_Z$ of the functional $\E_Z$ must lose compactness as $Z\to 0$.  We show this is indeed the case: for small $Z>0$, $\Om_Z$ is composed of a finite number of widely spaced disjoint compact components, separated by a distance on the order of $Z^{-1/(s-p)}$.  Moreover, we show that the components are arranged in a way which (after rescaling by $Z^{1/(s-p)}$) optimizes a discrete interaction energy,
\begin{equation}\label{eq:interactN}
\FNm(y_0,\dots,y_N):= \sum_{i,\, j=0 \atop i\neq j}^N {m^i\, m^j\over |y_i-y_j|^s}  
        -\sum_{i=1}^N \frac{m^i}{|y_i|^p},
\end{equation}
where $\underline m=(m^0,\dots,m^N)$ with $\sum_{i=0}^N m^i=M$, and $y_0,\dots, y_N$ in the admissible class
\beqn\label{eqn:admiss}
   \Sigma_N:=\{(y_0,\dots,y_N)\subset\R^{3(d+1)}: \ y_0=0\}.  
\eeqn

Our main result describes the structure of minimizers of $\E_Z$ for small $Z>0$:
\begin{theorem}\label{thm:Z20}
Let $\Om_Z$ be minimizers of $\E_Z$ for $Z>0$.  Then for any sequence $Z\to 0$ there exists a subsequence $Z_n\to 0$ so that either

\noindent {\bf (A)} \, there exists a set $E^0$ with $|E^0|=M$ which minimizes $e_0(M)$, for which 
$\Om_{Z_n}\to E^0$ globally, i.e., $\chi_{\Om_{Z_n}}\to\chi_{E^0}$ in $L^1(\R^d)$ as $n\to\infty$; or

\noindent {\bf (B)} \, there exist: 
\begin{enumerate}
\item $N\in \mathbb{N}$;
\item  $(m^0,\dots,m^N)$, $m^i>0$ with $\sum_{i=0}^N m^i=M$;
\item  $x_n^0,\ldots, x_n^N\in\Rd$, with $x_n^0=0$, and $|x^i_n|\to\infty$ for $i\neq 0$. and $|x_n^i-x_n^j|\to\infty$ for $i\neq j$ as $n\to\infty$;
\item  $E^0,\ldots,E^N$ compact sets of finite perimeter, with $|E^i|=m^i\neq 0$ for $i=0,\ldots,N$;
\end{enumerate}
such that $\Om_n:=\Om_{Z_n}$ satisfies the following:
	\begin{gather}
	\label{eq:Z20.0}
	\begin{gathered}
	 \partial^*\Om_n \in C^{1,\frac12}, \ \text{ and for fixed } R>0 \text{ such that all } E^i \subset B_R(0), \\ (\partial^*\Om_n- x^i_n)\cap B_R(0) \to \partial^* E^i \text{ in } C^{1,\alpha} \text{ for all } \alpha\in (0,1/2); \\
	 \end{gathered}  \\    
	\label{eq:Z20.1}
		\biggl| \Om_n \ \triangle \left[ E^0 \cup \bigcup_{i=1}^N (E^i + x^i_n)\right] 
		\biggr|\longrightarrow 0; \\
	\label{eq:Z20.2}
		E^i \text{ attains the minimum in (\ref{eqn:e_0}), i.e., } e_0(m^i)=\E_0(E^i), \ i=0,1,\ldots,N;  \\
	\label{eq:Z20.3}
	\left\{ \begin{gathered}
		Z_n^{{1\over s-p}} x^i_n \longrightarrow y_i \text{ as } n\to\infty, 
  				\ i=1,\ldots,N, \\
				\text{ where } (0,y_1,\dots,y_N) \text{ minimize  $\FNm$ over $\Sigma_N$.}
				\end{gathered}\right.
	\end{gather}
\end{theorem}
Here $\pt^*$ denotes the reduced boundary of a set.  By regularity theory of perimeter minimizing sets (and more generally, $(\omega,r)$-minimizers of perimeters) the topological boundary $\partial E^i$ differs from the reduced boundary by a set of small Hausdorff dimension, $\dim_{\mathcal{H}}(\partial E^i\setminus\partial^* E^i)\leq n-8$.

We note the distinction between the existence result in Theorem~\ref{thm:Z20} of $\E_Z$ and those of the Gamow functional:  for Gamow's model, minimizers only exist for small mass $M$, and must be connected. On the other hand, for $Z>0$ but small, minimizers of $\E_Z$ always exist for any $M$ but must be {\em  disconnected} for mass $M>m^*$.

The proof of Theorem~\ref{thm:Z20} relies on a  general concentration-compactness lemma  
(Lemma~\ref{lem:CC}) 
for minimizing sequences of $\E_Z$. We prove this result using a recent compactness result for sequences of Caccioppoli sets by Frank and Lieb \cite{FrLi2015}.   It is in this lemma that we first encounter the effect of splitting of the support of minimizers, when the total mass is large.  The resulting structure (as described by conclusions \eqref{eq:Z20.1} and \eqref{eq:Z20.2} of Theorem~\ref{thm:Z20},) was formalized by Kn\"upfer, Muratov, and Novaga \cite[Definition 4.3]{KnMuNo2016}); we adapt their definition to $\E_Z$:

\begin{definition}\label{def:genmin}
Let $Z\geq 0$ and $M>0$.  A \underbar{generalized minimizer} of $\E_Z$ is a finite collection $(E^0,E^1,\dots,E^N)$ of sets of finite perimeter, such that:
\begin{enumerate}
\item  $|E^i|:=m^i$, $i=0,1,\dots,N$, with $\sum_{i=0}^N m^i = M$;
\item  $E^0$ attains the minimum in $e_Z(m^0)$ and $E^i$ attains $e_0(m^i)$, $i=1,\dots,N$;
\item  $e_Z(M)= e_Z(m^0) + \sum_{i=1}^N e_0(m^i)$.
\end{enumerate}
\end{definition}
In \cite{KnMuNo2016} the authors prove the existence of generalized minimizers for the Gamow problem $Z=0$.  Here we improve their result: it follows immediately from the concentration lemma (Lemma 6) that {\it any} minimizing sequence of $\E_Z$, for $Z\geq 0$, is completely characterized (up to sets of vanishingly small measure, and along subsequences) by a generalized minimizer:

\begin{corollary} \label{thm:gen_min}  Let $Z\geq 0$, $M>0$, and suppose $\{\Om_n\}_{n\in\N}$ is any minimizing sequence for $e_Z(M)$.  Then, there is a subsequence, $N\geq 0$, and a generalized minimizer $(E^0,E^1,\dots,E^N)$ of $\E_Z$, with 
$$   \left\vert \Om_n \ \triangle \left[ E^0 \cup \bigcup_{i=1}^N (E^i + x^i_n)\right] 
		\right\vert \longrightarrow 0 ,  $$
for a sequence of translations $(x^i_n)_{n\in\N}^{i=1,\dots,N}$.
\end{corollary} 

In the context of generalized minimizers, Theorem~\ref{thm:Z20} asserts that the family $\Om_Z$ of minimizers of $\E_Z$ makes a particular selection of a generalized minimizer (the sets $\{E^i\}_{i=1,\dots,N}$ obtained in the theorem,) for the generalized liquid drop problem $\E_0$.  We note that the special choice of generalized minimizer obtained this way may not be canonical, in the sense of viscosity solutions in PDE;  the sets and the pattern they form as $Z\to 0$ depend on the choice of external potential.

The concept of generalized minimizers is a familiar one in applications of concentration compactness, and is intimately related to the notion of ``critical points at infinity'', introduced by Bahri \cite{Bahri89} in his study of existence of solutions for Yamabe-type equations and other PDE problems with loss of compactness.  (See \cite{StruweVM} for other contexts involving critical points or functionals ``at infinity''.)

In addition to the concentration-compactness structure given in Lemma~\ref{lem:CC}, the proof of Theorem~\ref{thm:Z20} requires an expansion of the energy $\E_Z$ up to the third-order term in $Z$ (see Remark~\ref{expansion} below). In order to establish this, we combine the compactness of a sequence of minimizers $\Om_Z$ with regularity results stemming from the classical regularity properties of the perimeter functional improving the error estimates in \cite{FrLi2015}.  Similar methods were employed in a previous paper \cite{AlBrChTo2016}, concerning concentration of droplets in a sharp interface model of diblock copolymers under confinement.

We note that the limiting finite dimensional energy $\FNm(y_0,\dots,y_N)$ (unlike its counterpart in \cite{AlBrChTo2016}) is not coercive, and so it is not clear {\it a priori} that minimizing sequences for this energy should not split, with some number of points diverging to infinity.
However, in Proposition~\ref{prop:finitedim} we will show that this finite dimensional discrete variational problem attains its minimizer for all choices of $N$ and the masses $\underline{m}$, a result which we will use in studying the limit $Z\to 0$ but which is itself of independent interest.  

In light of Theorem~\ref{thm:Z20}, it is natural to ask if the family of functionals $\E_Z$ has a second-order $\Gamma$-limit, involving generalized minimizers of the Gamow functional and the finite dimensional interaction energy $\FNm$.  Such a result would imply the existence of local minimizers for $\E_Z$, with small $Z>0$.  However, the method to prove Theorem~\ref{thm:Z20} uses regularity properties of minimizers in a fundamental way, and does not directly extend to the more general setting of $\Gamma$-convergence.

%\medskip

Finally, Bonacini and Cristoferi \cite[Theorem 2.11]{BoCr14} have shown that there exists a critical value $\bar s(d)$ of the power in the Riesz kernel such that if  $s\in (0, \bar s(d))$, then the minimizers of $e_0(M)$ (when they exist) must be balls.  In other words, for small $s$, the critical mass for existence exactly coincides with the critical value at which minimizers must be balls.  In this case, we have a near-complete description of minimizers for small $Z>0$, as a finite configuration of balls of {\em equal radius}:

\medskip

\begin{theorem}\label{thm:BC}
Assume $0<s<\bar s(d)$, and $0<p<s<d$.  Then, the sets $E^i$ appearing in Theorem~\ref{thm:Z20} are all balls with equal volume $m^i=M/(N+1)$, $i=0,1,\dots,N$.
\end{theorem}

The idea behind the proof of Theorem~\ref{thm:BC} is that each diverging component of a minimizer of $\E_Z$ inherits the same Lagrange multiplier, and so each element $E^i$ of the generalized minimizer ``at infinity'' satisfies the same Euler-Lagrange equation.  When the minimizers are balls, the radius is uniquely determined by the Lagrange multiplier.  As the first part of the argument holds for any values of $s,d,M$, we in fact conjecture that the equipartition of mass between the components of the generalized minimizers is true whether the minimizers are balls or not.

\medskip

%While the liquid drop model \eqref{eqn:LD} was initially
%posed to describe nuclear structures, the fact that it encapsulates a
%rather ubiquitous competition of short- and long-range effects expresses a {\it universality}, with the liquid drop model's
%phenomenology shared by many other systems operating at very different
%length scales: from the nuclear to nanoscale in condensed
%matter systems, to centimeter scale for fluids and autocatalytic
%reaction-diffusion systems, all the way to cosmological scales. 

The liquid drop model \eqref{eqn:LD} was introduced to describe nuclear structure.  In fact it appears in various other contexts (mathematical and physical) to describe systems with competition between short- and long-range effects on many scales, from the nuclear to nanoscale (in condensed matter systems), to centimeter scale (for fluids and autocatalytic
reaction-diffusion systems,) and even on cosmological scales.
In the original quantum context for the atomic nucleus, we do not know of any physical interpretation of such a background potential, even one of  Coulombic type ($p=1$).  
However, in the wider context (particularly the cosmological context), 
 consideration of super-Newtonian forces appears in several theories. In fact, the validity of Newton's law at long distances has been a longstanding interest in physics. As Finzi notes in \cite{Finzi1963} stability of clusters of galaxies implies stronger attractive forces at long distances than that predicted by Newton's law. Motivated by similar observations, in \cite{Milgrom1983} Milgrom introduced the modified Newtonian dynamics (MOND) theory which suggests that the gravitational force experienced by a star in the outer regions of a galaxy must be stronger than Newton's law (cf.  \cite{Bekenstein2004, Bugg2015,Milgrom2015}). 

%%%%%%%%%%%%%%%%%%%%%%%%%%%%%%%%%%%%%%%%
%%%%%%%%%%%%%%%%%%%%%%%%%%%%%%%%%%%%%%%%
%%%%%%%%%%%%%%%%%%%%%%%%%%%%%%%%%%%%%%%%
\section{Concentration-compactness and existence}
%%%%%%%%%%%%%%%%%%%%%%%%%%%%%%%%%%%%%%%%
%%%%%%%%%%%%%%%%%%%%%%%%%%%%%%%%%%%%%%%%
%%%%%%%%%%%%%%%%%%%%%%%%%%%%%%%%%%%%%%%%

In this section we prove the basic concentration-compactness structure of minimizing sequences for $\E_Z$. While this result could be adapted, for example, from the classical theory of Lions  (see \cite{Lions84} or Lemma~1 in the Appendix of \cite{Lions1987}), or from compactness results for minimizing clusters as in \cite[Chapter 29]{Maggi}, here we use a recent compactness result by Frank and Lieb \cite{FrLi2015} which is particularly well-suited for our purposes.  

We will say that a sequence of sets $E_n\to E$ {\it globally} in $\R^d$ if the measure of the symmetric difference $| E_n\triangle E|\to 0$.  We similarly say that $E_n\to E$
 {\it locally} if for every compact $K\subset\R^d$, $(K\cap E_n)\to (K\cap E)$ globally.  Global convergence is thus equivalent to convergence of the characteristic functions $\chi_{E_n}\to\chi_E$ in $L^1(\R^d)$, while local convergence is merely $L^1_{\loc}$ convergence of the characteristic functions.

\begin{lemma}\label{lem:CC}
Let $Z\in [0,\infty)$ be fixed, and $\{\Om_n\}_{n\in\N}$ a minimizing sequence for $e_Z(M)$.  Then there exists a subsequence such that either

\noindent {\bf (A)} \, there exists a set $E^0$ with $|E^0|=M$ which minimizes $e_Z(M)$, for which 
$\Om_{n}\to E^0$ globally, i.e., $\chi_{\Om_{n}}\to\chi_{E^0}$ in $L^1(\R^d)$ as $n\to\infty$; or

\noindent {\bf (B)} \, there exist $N\in \mathbb{N}$;  $\{x^1_n,\dots,x^N_n\}_{n\in\N}\subset\R^d$, with $|x^i_n|\to\infty$ and sets of finite perimeter $\{F_n^0,\dots,F_n^N, \Om_n^N\}_{n\in\N}$ such that $|x^i_n-x^j_n|\to\infty$, $i\neq j$; with
\beqn\label{eq:anyZ1}
\Om_n = F^0_n \cup \left[ \bigcup_{i=1}^N (F^i_n+x^i_n)\right] \cup \Om^N_n,
\eeqn
a disjoint union with components satisfying the following:
\begin{enumerate}
\item $\Om^N_n\to \emptyset$ and $F^i_n\to E^i$, globally in $\R^d$, with $m^i:=|E^i|> 0$ for all $i=1,\dots,N$ and $|E^0|>0$ for $Z>0$;
\item $\displaystyle 
    M=\sum_{i=0}^N |E^i| = \lim_{n\to\infty} \left(\sum_{i=0}^N |F^i_n| + |\Om^N_n|\right)$;
\item $E^i$ attain the minimimum for $e_0(m^i)$ for each $i=1,\dots,N$;
\item $E^0$ attains the minimum for $e_Z(m^0)$;
\item $\displaystyle e_Z(M) \geq e_Z(m^0) + \sum_{i=1}^N e_0(m^i)$.
\end{enumerate}
\end{lemma}

As mentioned in the introduction (see \cite[Definition 4.3]{KnMuNo2016},) the collection of sets $\{E^0,\dots,E^N\}_{n\in\N}$ are referred to as a \emph{generalized minimizer} of $\E_Z$ for any $Z\geq 0$.   
 Kn\"upfer, Muratov, and Novaga prove the existence of generalized minimizers for the case $Z=0$ by considering a truncation of the energy $\E_0$ and by obtaining density bounds for minimizers of the truncated energy (cf. \cite[Theorem 4.5]{KnMuNo2016}). Our approach in proving Lemma~\ref{lem:CC}  is more direct, and provides qualitative information about the structure of minimizing sequences that we exploit in Theorem \ref{thm:Z20}.  In particular, Corollary~\ref{thm:gen_min} follows, since $F^i_n\to E^i$ and \eqref{eq:anyZ1} then imply
$$  \lim_{n\to\infty} \left|  \Om_n \, \triangle  
  \left( E^0 \cup \bigcup_{i=1}^N (E^i+x^i_n)\right)\right| = 0.
$$

%\begin{proof}[Proof of Theorem \ref{thm:gen_min}]
%The result follows by applying Lemma \ref{lem:CC} with $Z=0$.
%\end{proof}

Before going back to Lemma \ref{lem:CC}, we need the following result to conveniently deal with the confinement term.

\begin{lemma}\label{lem:confinement}
Assume $A_n$ is a sequence of measurable sets with $|A_n|=M$ and $A_n\to 0$ locally (that is, $|A_n\cap K|\to 0$ for any compact $K$.)  Then,
$$  \lim_{n\to\infty} \int_{A_n} {1\over |x|^p}\, dx =0.  $$
\end{lemma}

The proof is an elementary exercise in real analysis, obtained by truncating $|x|^{-p}$ both vertically and laterally.

We also require the following subadditivity condition, which follows from the same arguments as Lemma~4 of \cite{LO2}:  for any values $0<m'<m$, and any $Z\geq 0$,
\beqn\label{eq:subadd}
e_Z(m)\leq e_Z(m')+e_0(m-m').
\eeqn

\bigskip

\begin{proof}[Proof of Lemma \ref{lem:CC}]  Let $Z\geq 0$ be fixed and $\Om_n$ a minimizing sequence for $e_Z(M)$.
We prove this lemma in several step.

\smallskip

\noindent {\bf Step 1:} \  {\sl Passing to the limit directly.} \
 By the compact embedding of $BV(\R^d)$ in $L^1_{\loc}(\R^d)$ (see e.g. \cite[Corollary 12.27]{Maggi}) there exists a subsequence and a set of finite perimeter $E^0\subset\R^d$ so that $\Om_n\to E^0$ locally, that is, $\chi_{\Om_n}\to \chi_{E^0}$ in $L^1_{\loc}(\R^d)$. (At this point, we admit the possibility that $|E^0|=0$, but in fact we will see in Step~3 that $|E^0|>0$.)
 
We claim that if the limit set $|E^0|=M$, then case {\bf (A)} holds and we are done. Indeed, since $\{\Om_n\}_{n\in\N}$ is locally convergent, a subsequence converges almost everywhere in $\R^d$.  In addition, the measures of the sets converge, that is, $|\Om_n|=M=|E^0|$, so by the Brezis-Lieb Lemma \cite[Theorem 1.9]{LiebLoss97} we may then conclude that (along a subsequence) $\Om_n\to E^0$ globally.  By the lower semicontinuity of the perimeter (see \cite[Proposition 4.29]{Maggi}) we have
	\[
	  \Per E^0 \leq \liminf_{n\to\infty}\Per \Om_n
	 \]
On the other hand, \cite[Lemma 2.3]{FrLi2015} implies that the nonlocal part is lower semicontinuous, as well, that is,
	\[
		\mathcal{D}(E^0,E^0) \leq \liminf_{n\to\infty} \mathcal{D}(\Om_n,\Om_n) \qquad \text{where} \qquad
		\mathcal{D}(E,F) := \int_{E} \! \int_{F} \frac{dxdy}{|x-y|^s}.
	\]
To pass to the limit in the confinement term, we apply Lemma~\ref{lem:confinement} to the sequence $(\Om_n\setminus E^0)\to \emptyset$ locally, and together with the above we have
	\[
		\E_{Z}(E^0) \leq \liminf_{n\to\infty} \E_{Z}(\Omega_n).
	\]
Therefore we conclude that $E^0$ attains the minimum value of $\E_Z$, and the proof is complete for case {\bf (A)}.  

In the following we may thus assume that $m^0:=|E^0|<M$.

\medskip

\noindent {\bf Step 2:} \  {\sl Concentration-compactness.} \
In case $m^0=|E^0|<M$, by \cite[Lemma 2.2]{FrLi2015} (with no translation necessary, i.e., $x_n^0=0$) there exist radii $r^0_n\in (0,\infty)$ such that for
	\[
		   F^0_n = \Omega_n\cap B_{r^0_n}(0) \qquad G^0_n = \Omega_n\setminus \ol{B_{r^0_n}}(0)
	\]
where $F^0_n\to E^0$ globally, $G^0_n\to \emptyset$ locally as $n\to\infty$ with $m^0_n:=|F^0_n|\to |E^0|=m^0<M$,
and 
	\beqn\label{eq:split1}
		\lim_{n\to\infty} \big(\Per(\Omega_n) -\Per(F^0_n) - \Per(G^0_n)\big)=0, \qquad \liminf_{n\to\infty} \Per(F^0_n) \geq \Per(E^0).
	\eeqn
In addition, again by \cite[Lemma 2.3]{FrLi2015},
	\beqn \label{eq:split2}
		\mathcal{D}(\Omega_n,\Omega_n) 
		  = \mathcal{D}(F^0_n,F^0_n) + \mathcal{D}(G^0_n,G^0_n) + o(1) 
		  = \mathcal{D}(E^0,E^0) + \mathcal{D}(G^0_n,G^0_n) + o(1).
	\eeqn
Finally, by Lemma~\ref{lem:confinement}, the confinement term is absent for $G^0_n$, which tends to zero locally.
In conclusion, we have a splitting of the energy, 
	\beqn \label{eq:split3}
		\E_{Z}(\Omega_n) = \E_{Z}(E^0) + \E_{0}(G^0_n) + o(1).  
	\eeqn

We define $\Omega^0_n := G^0_n$, with $|\Omega^0_n|=M-m^0_n=M-m^0+o(1)>0$, and begin an iterative process of locating escaping concentrations of mass, as in the concentration-compactness lemma of Lions (cf. \cite{Lions84}).   
By \cite[Proposition 2.1]{FrLi2015}, there is a set $E^1$ of positive measure and a sequence of points $x^1_n\in\Rd$ for which $\Omega^0_n - x^1_n\to E^1$ locally.
Since $\Om_n^0\to 0$ locally, it follows that $|x_n^1|\to\infty$.
  In addition, $|E^1|\in (0,M-m^0]$ and $\Per(E^1)\leq \liminf_{n\to\infty}\Per(\Omega^0_n)$.  In case of nonuniqueness of such translates, we define 
	\begin{multline}
	  \mu(\{\Omega^0_n\}):= \sup\big\{ |A| \colon \text{there exist } A\subset\Rd \text{ and } \\ \{\xi_n\}\subset\Rd \text{ such that } \Omega^0_n-\xi_n \to A \text{ locally}\big\}. \nonumber
	\end{multline}
We may thus choose $\{x^1_n\}$ and $E^1$ such that $|E^1|>\frac12\mu(\{\Omega^0_n\})$.

Applying \cite[Lemma 2.2]{FrLi2015} as above, there exist radii $r^1_n\in (0,\infty)$ such that if we define
	\[
		   F^1_n = (\Omega^0_n-x_n^1)\cap B_{r^1_n}(0) \qquad G^1_n = (\Omega^0_n-x_n^1)\setminus \ol{B_{r^1_n}}(0)
	\]
then $F^1_n\to E^1$ globally, $G^1_n\to \emptyset$ locally as $n\to\infty$, with $m^1_n:=|F^1_n|\to |E^1|=:m^1\in (0,M-m^0]$,
	\begin{multline} \label{eq:split4}
			0=\lim_{n\to\infty} \big(\Per(\Omega^0_n) -\Per(F^1_n) - \Per(G^1_n)\big) \\
				=	\lim_{n\to\infty}\big(\Per(\Omega_n) -\Per(F^0_n) -\Per(F^1_n) - \Per(G^1_n)\big),
	\end{multline}
and $\liminf_{n\to\infty} \Per(F^1_n) \geq \Per(E^1)$.
Finally, by \cite[Lemma 2.3]{FrLi2015},
\begin{align}\nonumber
\mathcal{D}(\Omega_n, \Omega_n) 
&= \mathcal{D}(F^0_n,F^0_n)+\mathcal{D}(\Omega^0_n,\Omega^0_n) \\
\nonumber
&=\mathcal{D}(F^0_n,F^0_n)+ \mathcal{D}(F^1_n,F^1_n) + \mathcal{D}(G^1_n,G^1_n) + o(1) \\
&=\mathcal{D}(E^0,E^0)+ \mathcal{D}(E^1,E^1) + \mathcal{D}(G^1_n,G^1_n) + o(1).
\label{eq:split5}
\end{align}
In particular, 
\begin{equation}\label{eq:split6}
  \E_{Z}(\Omega_n) \geq \E_{Z}(E^0) +\E_{0}(E^1) + \E_{0}(G^1_n) + o(1).  
\end{equation} 

If $|G^1_n|\to 0$, the process terminates with $N=1$.  If not, we let $\Omega^1_n:=G^1_n+x_n^1$, and repeat the above procedure with $\mu(\{\Omega^1_n\})\in (0,M-m^0-m^1]$, iteratively generating an at most countable collection of concentration sets $F^i_n\to E^i$ and remainder sets $\Omega_n^i$, $i=1,2,\ldots$, satisfying 
\begin{gather}
\Omega_n^{i-1}=[F^i_n+ x^i_n]\cup \Omega_n^i, \quad \text{a disjoint union,} \\
\label{eq:cc0}
\Om_n =\Omega^k_n \cup \left[\bigcup_{i=0}^k (F^i_n + x^i_n)\right], \qquad |x_n^i-x_n^j|\to\infty, \ i\neq j,\\
\label{eq:cc1}
M = \sum_{i=0}^k m_n^i + \lim_{n\to\infty} |\Omega^k_n|=\sum_{i=0}^k m^i + \lim_{n\to\infty} |G^k_n|, \\
\label{eq:cc2}
m^k\geq\frac12 \mu(\{\Omega^{k-1}_n\}), \\
\label{eq:cc3}
\E_{Z}(\Omega_n) \geq \E_{Z}(\Omega^k_n)+\E_Z(E^0)+\sum_{i=1}^k \E_0(E^i) + o(1)
\end{gather}
for $k\in\N$.  We note that the decomposition in \eqref{eq:cc0} is disjoint, with $\Om^k_n\to \emptyset$ locally.

In case $|\Omega^N_n|\to 0$ for some finite $N\in\N$, the process terminates and the decomposition is finite.  If the number of components $E^i$ is countable, then by \eqref{eq:cc1} we must have $m^i\to 0$ as $i\to\infty$, and hence $\mu(\{\Omega^k_n\})\to 0$ as $k\to\infty$, by \eqref{eq:cc2}.  We may then conclude that the iteration exhausts all of the mass, and
	\beqn \label{eq:samemass}
		M= \sum_{i=0}^\infty m^i.  
	\eeqn

\medskip

\noindent {\bf Step 3:} \ {\sl If $Z>0$, then $|E^0|\neq 0$.} \  Suppose the contrary, i.e., that $|E^0|=0$.  Define $\tilde\Om_n:=\Om_n- x_n^1$, and so by the above construction $\tilde\Om_n\to E^1$ and $\tilde\Om_n\setminus F^1_n=G_n^1\to \emptyset$ locally.  Thus, by Lemma~\ref{lem:confinement}, for any $i\neq 1$,
$$  \lim_{n\to\infty} \int_{\tilde\Om_n\setminus F^1_n} {1\over |x|^p}dx =0.  $$
Since the perimeter and nonlocal terms in $\E_Z$ are translation invariant, we arrive at
$$  \E_Z(\tilde\Om_n) - \E_Z(\Om_n) = -Z \int_{F_n^1} {1\over |x|^p}dx + o(1) 
=-Z \int_{E^1} {1\over |x|^p}dx +o(1) <0,
$$
a contradiction.

\medskip

\noindent {\bf Step 4:} \ {\sl The sets $E^i$ are minimal, and there are finitely many.} \
%Using the sequences of masses $m^i$ and translates $x_n^i$ obtained in Step 1, we begin by constructing a primative upper bound on the energy.  Let $\epsilon>0$ be given.  To this end, choose a set $K\subset\Rd$ of finite perimeter with $|K|=M$ and $\E_Z(K)<e_Z(M)+ \epsilon$.  Then, $K$ is an admissible set for $\E_{Z}$, so $e_{Z}(M) \leq \E_{Z}(K) \leq e_Z(M) + \epsilon $.
 
By Lemma~\ref{lem:confinement},
$$  \liminf_{n\to\infty}
   \E_Z(\Om_n^k) = \liminf_{n\to\infty} \left[\E_0(\Om_n^k) - Z\int_{\Om_n^k}  |x|^{-p}\, dx
   \right]
   =  \liminf_{n\to\infty} \E_0(\Om_n^k)\geq 0.  $$
Thus, as \eqref{eq:cc3} holds for all $k\in\N$, we have:
we then have
$$  \E_{Z}(\Omega_n) \geq \E_Z(E^0)+\sum_{i=1}^\infty \E_0 (E^i) -o(1). $$
We may then conclude,
	\begin{align*}
		e_Z(M)+o(1)\geq \E_{Z}(\Omega_n) &\geq \E_Z(E^0)+\sum_{i=1}^\infty \E_0 (E^i) 
				-o(1) \\
		 &\geq
         e_Z(m^0)+ \sum_{i=1}^\infty e_0(m^i)-o(1)  \\
         &\geq e_Z(M)-o(1),
    \end{align*}
by the subadditivity condition \eqref{eq:subadd} of $e_Z$.  Matching the upper and lower bounds we have,
$$   (\E_Z(E^0)-e_Z(m^0))+ \sum_{i=1}^\infty [\E_0(E^i)-e_0(m^i)] \leq 0.  $$
Since each term is nonnegative, each must be zero, and so each set $E^i$, $i=0,1,\dots$ is minimal.

Lastly, as the series converges we must have $e_0(m^i)\to 0$ as $i\to\infty$, and from this fact we may conclude that only finitely many of $m^i$ are nonzero. This follows almost verbatim as in \cite[Lemma 4.4]{ChPe2010}, so we sketch the main idea here for completeness. Now let $m^{**}>0$ be the constant such that $e_0$ is attained uniquely by a ball of volume $m$ for $m\leq m^{**}$ (cf. \cite[Theorem 1.3]{FiFuMaMiMo2015}). For the ball, the value $e_0(m)=C_1 m^{(d-1)/d} + C_2 m^{(2d-s)/d}$ is explicitly known (with universal constants $C_1,\, C_2$), and is strictly concave when $m<\trns{m}:=\min\{m^{**},(C_1/C_2)^{d/(1+d-s)}\}$.  In particular, it follows that if the masses $\{m^i\}$ minimize the expression $e_0(M-m^0)=\sum_{i=1}^\infty e_0(m^i)$ then there can be at most one $m^i\in (0,\trns{m})$. Hence, there can only be a finite number of components $E^i$.

This completes the proof of the concentration lemma. 
\end{proof}

\bigskip

The proof of Theorem~\ref{thm:anyZ} is essentially given in \cite{AlBrChTo2017_1} for the Newtonian case $s=1$ and for more general confinement terms, but we include a short proof here for completeness.

\medskip

\begin{proof}[Proof of Theorem \ref{thm:anyZ}]  We apply Lemma~\ref{lem:CC} to any  minimizing sequence $\{\Om_n\}$ for $e_Z(M)$.  If case {\bf (A)} holds, the sequence converges to a minimizer and we are done.  So assume there is splitting as in case {\bf (B)}, and so there exists $N\in\N$, sets $E^i\subset\R^d$ with $|E^i|=m^i\neq 0$ for each $i=0,1,\dots,N$, $M=\sum_{i=0}^N m^i$, satisfying the lower bound,
\beqn\label{lowbd}
e_Z(M) \geq e_Z(m^0) + \sum_{i=1}^N e_0(m^i)
\eeqn

We now construct a better upper bound, using the slow decay rate of the potential (recall that $0<p<s$).  As each $E^i$ is a minimizer, it is essentially bounded (cf. \cite[Lemma 4.1]{KnMu2014}). Hence, we may choose a representative for $E^i$ such that, for some $R>0$, we have $E^i \subset B_R(0)$ for all $i=0,1,\ldots,N$.  For $i=1,\ldots,N$ let $b_i\in\Rd$ such that $|b_i|=1$, and let $b_0=0$. Define
	\[
		\Om_t := E^0\cup\left[\bigcup_{i=1}^N (E^i + tb_i)\right].
	\]
Note that for all sufficiently large $t$ the sets are disjoint, and so using the translation invariance of the perimeter and the nonlocal part $\mathcal{D}$, we have
\begin{align}\label{upbd}
	e_Z(M)&\leq \E_Z(\Om_t) =  \E_Z(E^0) + \sum_{i=1}^N \E_0(E^i) +\FF(t)-\GG(t) \\
	 \nonumber
	 &= e_Z(m^0) + 	\sum_{i=1}^N e_0(m^i) +\FF(t)-\GG(t),
\end{align}
where
$$  \FF(t):= \sum_{\substack{i,j=0\\ i\neq j}}^N \int_{E^i+tb_i}\int_{E^j+tb_j} \frac{dxdy}{|x-y|^s} \ \ \text{and} 
   \ \ \GG(t):=\sum_{i=1}^N\int_{E^i+tb_i} \frac{dx}{|x|^p}.  $$

We now estimate each; first, we claim there is a $t_0>0$ for which $\FF(t)\leq Ct^{-s}$ for all $t>t_0$.  Indeed, for any $i\neq j$, with the change of variables $t\xi=x$, $t\eta=y$, we have
\begin{align*}
t^s\int_{E^i+tb_i}\int_{E^j+tb_j} \frac{dxdy}{|x-y|^s}
&\leq t^s \int_{B_R(tb_i)} \int_{B_R(tb_j)} \frac{dxdy}{|x-y|^s} \\
& = {|B_R|^2 \over |B_{R/t}|^2}\int_{B_{R/t}(b_i)} \int_{B_{R/t}(b_j)} 
\frac{d\xi d\eta}{|\xi-\eta|^s} \longrightarrow  {|B_R|^2\over |b_i-b_j|^s},
\end{align*}
as $t\to\infty$.  There are only finitely many terms in $\FF(t)$, and so the claim holds.

To estimate $\GG(t)$ from below, we note that as $t\to\infty$,
$$  t^{-p}\int_{E^i+tb_i} |x|^{-p} \, dx = \int_{E^i} \left|b_i+{x\over t}\right|^{-p}\, dx
     \longrightarrow |E^i|=m^i,
$$
by dominated convergence.  Thus, $\FF(t)-\GG(t)\leq c_1 t^{-s} - MZ t^{-p}<0$ for sufficiently large $t$, and thus \eqref{upbd} is in contradiction with \eqref{lowbd}.
Thus we must have $|\Om^0|=M$ and $e_Z(M)=\E_Z(E^0)$, for any $M>0$ and for any $Z>0$. 
\end{proof}

%%%%%%%%%%%%%%%%%%%%%%%%%%%%%%%%%%%%%%%%
%%%%%%%%%%%%%%%%%%%%%%%%%%%%%%%%%%%%%%%%
%%%%%%%%%%%%%%%%%%%%%%%%%%%%%%%%%%%%%%%%
\section{The limit $Z\to 0$}
%%%%%%%%%%%%%%%%%%%%%%%%%%%%%%%%%%%%%%%%
%%%%%%%%%%%%%%%%%%%%%%%%%%%%%%%%%%%%%%%%
%%%%%%%%%%%%%%%%%%%%%%%%%%%%%%%%%%%%%%%%

We start this section by proving that the finite dimensional energy functional $\FNm$ given by \eqref{eq:interactN} has a minimizer.
We define 
$$\mu_{N,\underline{m}}:=\inf_{\Sigma_N} \FNm,  $$ 
where the admissible set $\Sigma_N$ is defined in \eqref{eqn:admiss}.

\begin{proposition}\label{prop:finitedim}
For any $N\in \N$ and $\underline m$, the functional $\FNm$ attains its minimum
$\mu_{N,\underline{m}}<0$ on the admissible class $\Sigma_N$.
\end{proposition}
%%%%
\begin{proof}
Consider any minimizing sequence $\{x_n^i\}_{n\in\N}$, $i=1,\ldots,N$, in $\Sigma_N$, that is, 
$\mu_{N,\underline{m}}=\lim_{n\to\infty} \FNm(0,x^n_1,\dots,x^n_N)$.  If all the sequences $\{x_n^i\}_{n\in\N}$, $i=1,\dots,N$, remain bounded, then we obtain convergence to a minimizer along some subsequence.  So instead, assume that there is an integer $k\in \{0,1,\dots,N-1\}$ and a subsequence (not relabelled) so that 
	\beqn \label{eq:split}
			\left\{ \begin{aligned}
					x_n^i \underset{n\to\infty}{\longrightarrow} a_i, \quad \forall \ i=0,\dots,k, \ \text{but} \\
					|x^i_n| \underset{n\to\infty}{\longrightarrow} \infty, \quad \forall \ i=(k+1),\dots,N.
						\end{aligned} \right.
	\eeqn
We first treat the case where $k\geq 1$.  Decompose $\FNm$ into pieces,
	\begin{multline} \label{eq:decomp}
  \FNm(0,x^1_n,\dots,x^N_n)  = \mathbf{F}_{k,(m^0,\dots,m^k)}(0,x^1_n,\dots,x^k_n) \\ + 
    \mathbf{F}_{N-k,(m^{k+1},\dots,m^N)}(x^{k+1}_n,\dots,x^N_n) + I_{k,N}, 
	\end{multline}
with interaction term between the two families,
	\[
		  I_{k,N} = \sum_{i=0}^k \sum_{j=k+1}^N {m^i m^j\over |x_n^i - x_n^j|^s}.
	\]
Using the splitting \eqref{eq:split}, we have 
	\beqn \label{eq:tocontradict}
	\begin{aligned} 
		\mu_{N,\underline{m}} &\geq \liminf_{n\to\infty} \left[\mathbf{F}_{k,(m^0,\dots,m^k)}(0,x^1_n,\ldots, x^k_n) + \sum_{i,j=k+1\atop j\neq i}^N {m^i m^j\over |x_n^i - x_n^j|^s}\right]  \\
		&\geq \liminf_{n\to\infty} \mathbf{F}_{k,(m^0,\dots,m^k)}(0,x^1_n,\dots, x^k_n)  \\
		&= \mathbf{F}_{k,(m^0,\dots,m^k)}(0,a_1,\dots, a_k).
	\end{aligned}
	\eeqn

To obtain a contradiction to \eqref{eq:tocontradict}, we define a new configuration given by the points
$\{a_1,\dots,a_k, Ry_1,\dots, Ry_{N-k}\}$ with $\{y_1,\dots,y_{N-k}\}$ distinct points on the unit sphere $|y_j|=1$, and $R>0$ to be determined.  By the same decomposition as in \eqref{eq:decomp},
	\begin{multline}\label{eq:ydecomp}
  		\FNm(0,a_1,\dots,a_k, Ry_1,\dots, Ry_{N-k}) = 
  			\mathbf{F}_{k,(m^0,\dots,m^k)}(0,a_1,\dots,a_k) \\ + 
    		\mathbf{F}_{N-k,(m^{k+1},\dots,m^N)}(Ry_1,\dots, Ry_{N-k}) + \tilde I_{k,N},
	\end{multline}
with $\tilde I_{k,N}$ representing the interaction terms.  If $|a_i|<R_0$ for some $R_0>0$ and for each $i=1,\dots,k$, and if $R>2R_0$, the interaction terms may be estimated by
	\[
	   \tilde I_{k,N} \leq C_1(k,N,\underline{m}) R^{-s}.
	\]
Similarly, since $|Ry_i-Ry_j|\geq C_2 R$, $i\neq j$, for some constant $C_2>0$, we also have
	\[
	  \mathbf{F}_{N-k,(m^{k+1},\dots,m^N)}(Ry_1,\dots, Ry_{N-k})\leq \sum_{i,j=1\atop i\neq j}^{N-k} {m^{k+i}m^{k+j}\over |Ry_i-Ry_j|^s} 
     \leq C_3(k,N,\underline{m}) R^{-s}.
    \]
On the other hand, 
	\[
		 \sum_{i=1}^{N-k} m^{k+i}|Ry_i|^{-p} = R^{-p}\sum_{i=1}^{N-k} m^{k+i} 
   \geq C_4(k,N,\underline{m})R^{-p}.
	\]
and thus \eqref{eq:ydecomp} yields,
	\beqn \label{eq:kge1}
	\begin{aligned}
 		\mu_{N,\underline{m}} &\leq \FNm(0,a_1,\dots,a_k, Ry_1,\dots, Ry_{N-k})\\
 					& \leq \mathbf{F}_{k,(m^0,\dots,m^k)}(0,a_1,\dots,a_k) - C_4(k,N,\underline{m})R^{-p} + O(R^{-s}) \\
     				&< \mathbf{F}_{k,(m^0,\dots,m^k)}(0,a_1,\dots,a_k),
	\end{aligned}
	\eeqn
for $R>R_0>0$ chosen large enough, contradicting \eqref{eq:tocontradict} in case $k\geq 1$.
For $k=0$, that is, if $|x^i_n|\to\infty$ for each $i=1,\dots,N$, we note that
$$  \mu_{N,\underline{m}} \geq \liminf_{n\to\infty} \sum_{i,j=0\atop j\neq i}^N {m^i m^j\over |x_n^i - x_n^j|^s}\geq 0, $$
while the same construction which produced \eqref{eq:kge1} yields the contradictory estimate $\mu_{N,\underline{m}}<0$.
   In conclusion, the entire minimizing sequence must remain bounded, and so the minimum is attained. 
\end{proof}

Next we show that the infimum of the regularized energies $\E_Z$ converge to the infimum of $\E_0$.

\begin{lemma}\label{lem:min}
$\lim_{Z\to 0} e_Z(M) = e_0(M)$.
\end{lemma}
\begin{proof}
Let $\Om_Z$ be a minimizer of $\E_Z$ which exists for any $Z>0$ and $M>0$ by Theorem \ref{thm:anyZ}. Then, clearly $e_Z(M) \leq e_0(M)$ for all $Z>0$, and
\begin{align*}
  \E_0(\Om_Z)&= \E_Z(\Om_Z) + Z\int_{\Om_Z} {dx\over |x|^p}\leq \E_Z(\Om_Z) + Z\int_{B_1(0)} \frac{dx}{|x|^p} \\
  &\qquad\qquad\qquad\qquad\qquad\qquad\qquad\qquad\qquad + Z |\Om_Z \cap (\Rd\setminus B_1(0))| \\
  &\leq \E_Z(\Omega_Z) + \left({\omega_d \over (d-p)} + M\right)Z,
\end{align*}
where $\omega_d=|B_1(0)|$ denotes the volume of the unit ball in $\Rd$. Therefore we also have $e_0(M)\leq \liminf_{Z\to 0} e_Z(M)$, which proves the claim. 
\end{proof}

%It will be convenient to refer to the minimizing sets $\Omega_Z$, with $u_Z=\chi_{\Omega_Z}$ minimizers of $\E_Z$.  Define
%$$  \E_Z(E)= \E_Z(\chi_E)  $$
%for sets $E$ of finite perimeter.

The following lemma is key in obtaining regularity properties for a family of minimizers of the functionals $\E_Z$.

\begin{lemma}\label{lem:omegamin}
The family of minimizers $\{\Om_Z\}_{Z\in (0,1]}$ of $\E_Z$ are $(\omega,r)$-minimizers of the perimeter functional in $\mathcal{O}:=\Rd\setminus \overline{B_\delta(0)}$ for any $\delta>0$, with $\omega,\, r>0$ uniformly chosen for $Z\in (0,1]$; that is,
	\[
		 \Per(\Omega_Z)\leq \Per(F) + \omega\, | \Omega_Z\triangle F|,
	\]
for all $F\subset\Rd$ with $\Omega_Z\triangle F\subset B_r(x_0)\subset \R^d\setminus B_\delta(0)$.
\end{lemma}
\begin{proof}
First we show that the constraint $|\Omega_Z|=M$ may be replaced by a penalization, following \cite[Theorem 2.7]{BoCr14} (see also \cite[Section 2]{EsFu2011}.)  For $\lambda>0$ (to be determined), define the penalized functionals
	\[
	   \calF_Z^\lambda(F) := \E_Z(F) + \lambda\bigl| |F|-|\Omega_Z|\bigr| =
		    \E_Z(F) + \lambda\bigl| |F|-M\bigr|.
	\]

We claim that there exists $\lambda>0$ so that for all $0<Z\leq 1$, 
\begin{equation}\label{eq:omega1}
\min \calF_Z^\lambda = \calF_Z^\lambda(\Omega_Z)= \E_Z(\Omega_Z),
\end{equation}
i.e., the unconstrained minimizer of $\calF_Z^\lambda$ coincides with the mass-constrained minimizer of $\E_Z$.
Indeed, the existence of a constant $\lambda=\lambda_Z>0$ for each fixed $Z>0$ satisfying the claim follows by a minor modification of \cite[Theorem 2.7]{BoCr14}, so it suffices to show that $\lambda$ may be chosen independently of $Z$.  Suppose no such $\lambda$ exists, so there are sequences $Z_n\to 0$, $\lambda_n\to\infty$, and sets $E_n\subset\Rd$, $|E_n|\neq M$, with 
$\calF_{Z_n}^{\lambda_n}(E_n)<\calF_{Z_n}^{\lambda_n}(\Omega_{Z_n})$.  Note that $\lambda_n\to\infty$ implies that $|E_n|\to M$.  

Define sets $\tilde E_n=t_n E_n$ where $t_n=[M/|E_n|]^{1/d}$, so $|\tilde E_n|=M$.  Each term in $\calF_{Z_n}(\tilde E_n)$ may then be calculated via scaling,
\begin{align*}  \calF_{Z_n}^{\lambda_n}(\tilde E_n) &= \E_{Z_n}(\tilde E_n) = t_n^{d-1}\Per(E_n) 
       +t_n^{2d-s}\mathcal{D}(E_n,E_n) 
       - t_n^{d-p}Z_n\int_{E_n} |x|^{-p}\, dx\\
    &=    \calF_{Z_n}^{\lambda_n}(E_n) + \left(t_n^{d-1}-1\right)\Per(E_n)
                      + \left(t_n^{2d-s}-1\right)\mathcal{D}(E_n,E_n) \\
    &\qquad
                        -\left(t_n^{d-p}-1\right)Z_n\int_{E_n}|x|^{-p}\, dx 
                         - \lambda_n\left|t_n^{d-1}-1\right|\,|E_n| \\
     &\leq \calF_{Z_n}^{\lambda_n}(E_n) + \left|t_n^{d-1}-1\right|\,|E_n|\, 
     \left[ \E_{0}(E_n){ (t_n^{d-1}+ t_n^{2d-s} -2)                        \over
           \left|t_n^{d-1}-1\right|\,|E_n|} -\lambda_n \right] \\
      &< \calF_{Z_n}^{\lambda_n}(E_n),
\end{align*}
as $\lambda_n\to\infty$ since the term in brackets is eventually negative.
 This contradicts the definition of $E_n$ as minimizers of $\calF_{Z_n}^{\lambda_n}$, and so we conclude that \eqref{eq:omega1} must hold.

\medskip

Now fix any $r>0$ and assume $B_r(x_0)\cap B_\delta(0)=\emptyset$, and $F\subset\Rd$ with $\Omega_Z\triangle F\subset B_r(x_0)$.  Denote
	\[
	  	\mathcal{V}(F)\, :=\, \int_F \frac{dx}{|x|^p}.
	\]
Then, $\E_Z(\Omega_Z)=\calF_Z^\lambda(\Omega_Z)\leq \calF_Z^\lambda(F)$ implies that
\begin{align*}
\Per(\Omega_Z)  &\leq \Per(F) + \big(\mathcal{D}(F,F)-\mathcal{D}(\Omega_Z,\Omega_Z)\big)
     + \big(\mathcal{V}(\Omega_Z)-\mathcal{V}(F)\big) + \lambda\bigl| |F| - M\bigr|\\
     &\leq  \Per(F) + (C_0 + \delta^{-p} + \lambda)|\Omega_Z\triangle F|,
\end{align*}
where the difference of the nonlocal terms is estimated in \cite[Proposition 2.3]{BoCr14}, and to estimate the confinement term we use the fact that $|x|^{-p}\in L^{\infty}(\Rd\setminus B_\delta(0))$.  Thus, $\Omega_Z$ are $(\omega,r)$-minimizers of the perimeter functional in $\Rd\setminus B_\delta(0)$ with $\omega=C_0+\delta^{-p}+\lambda$ and any $r>0$. 
\end{proof}

Finally, we state the following regularity results for $(\omega,r)$-minimizers that we will require in the proof of Theorems \ref{thm:Z20} and \ref{thm:BC}.

\begin{lemma}[see Theorems 21.8, 21.14 and 26.6 of \cite{Maggi}]\label{thm:tamanini}  Let $\mathcal{O}\subset\R^d$ be an open set.
\begin{enumerate}
\item  If $E\subset\R^d$ is an $(\omega,r)$-minimizer of perimeter in $\mathcal{O}$ then $\partial^*E\cap\mathcal{O}$ is a $C^{1,\alpha}$ hypersurface for any $\alpha\in(0,1/2)$.
\item  If $E_n\subset\R^d$ is a sequence of uniformly $(\omega,r)$-minimizers of perimeter in $\mathcal{O}$ with $E_n\to E_\infty$ locally in $\mathcal{O}$, then for any sequence $x_n\in \partial E_n$ with $x_n\to x_\infty$ we have $x_\infty\in\partial E_\infty$.  Moreover, if $x_n\in \partial^* E_n$, then $x_\infty\in \partial^* E_\infty$ and the normal vectors satisfy $\nu(x_n)\to \nu(x_\infty)$.
\end{enumerate}
\end{lemma}

Thus, a sequence of uniformly $(\omega,r)$-minimizers of perimeter which converges locally has its reduced boundary convergent in the Hausdorff metric. We remark that a stronger form of this $C^{1,\alpha}$ convergence of $\partial^* E_n\to \partial^* E$ is stated in \cite[Theorem 4.2]{AcFuMo13}:  under the hypothesis that $E_n\to E$ globally in $\mathcal{O}$, in fact the convergence of the boundaries is in $C^{1,\alpha}$ for $\alpha\in(0,1/2)$, and $\partial E_n$ may be realized as a $C^{1,\alpha}$ graph over $\partial E$.

We remark that we only need the full force of the regularity theory for Theorem~\ref{thm:BC}. For the proof of Theorem~\ref{thm:Z20} we only require that minimizers for $\E_Z$ are supported in compact sets and converge pointwise to the disjoint components $E^i$.

\bigskip

Now we are ready to prove our main results.

\begin{proof}[Proof of Theorem~\ref{thm:Z20}]  
Let $\{\Om_n\}_{n\in\N}$ with $\Om_n:=\Om_{Z_n}$ be a sequence of minimizers for $e_{Z_n}$ with $Z_n\to 0$. By Lemma \ref{lem:min}, $\{\Om_n\}$ form in fact a minimizing sequence for $e_0$. Therefore by Lemma \ref{lem:CC} we obtain either {\bf (A)} or assertions (i), (ii), and \eqref{eq:Z20.1}, \eqref{eq:Z20.2} in (iii) of part {\bf (B)} of Theorem \ref{thm:Z20}. The statement \eqref{eq:Z20.0}, on the other hand, follows directly from Lemmas \ref{lem:omegamin} and \ref{thm:tamanini} (or \cite[Theorem 4.2]{AcFuMo13}.)
In order to prove \eqref{eq:Z20.3} we adopt the notations from Lemma~\ref{lem:CC}. Our goal here is to use the regularity of minimizing sets to improve the precision of the lower bound defined in the concentration lemma. We prove this in several steps.

\smallskip

\noindent {\bf Step 1:} \ {\sl A more refined decomposition.} \ We return to Step 1 in the proof of Lemma~\ref{lem:CC}, and use the uniform $(\omega,r)$-minimality to show that 
	\[
		  \Omega_n=F^0_n\cup \left[\bigcup_{i=1}^k (F^i_n + x_n^i)\right] ,
	\]
splits cleanly, with no $o(1)$ error in the perimeter, and with remainder set $\Om_n^N=\emptyset$.  In particular, we claim that
	\beqn \label{eq:split3per}
		\Per(\Omega_n) = \sum_{i=0}^N \Per (F^i_n) 
%		\geq \sum_{i=0}^k \Per (F^i_n),
	\eeqn 
holds for each sufficiently large $n$.  For convenience, we define
	\[
	  \tilde F_n^i=F_n^i+x_n^i \qquad \text{and}\qquad  \hat\Omega_n^i=\Omega_n^{i-1} -x_n^i, \qquad i=0,1,\ldots,N.
	\]

To verify \eqref{eq:split3per}, we first note that $E^i$ being minimizers of $e_0(m^i)$, they are essentially bounded domains with smooth $\partial^* E^i$ (cf. \cite[Proposition 2.1 and Lemma 4.1]{KnMu2014}).  Therefore, we may fix $R>0$ so that a representative of each $E^i\subset B_{R/2}(0)$ for each $i=0,1,\dots,N$.  We observe that, since each $E^i$ is bounded, when defining $F^i_n=\hat\Omega^{i-1}_n \cap B_{r_n}(0)$ we may choose the radii $r_n$ found in \cite[Lemma 2.2]{FrLi2015} so that $r_n\in (R,2R)$.  As $\hat\Omega^i_n \to E^i$ locally, it converges globally in $\mathcal{O}:=B_{2R}(0)$.  For $i=1,\dots,N$, we invoke Lemma~\ref{lem:omegamin} which ensures that $\hat\Omega^i_n$ is a family of uniformly $(\omega,r)$-minimizers in $\mathcal{O}$.  By part (ii) the regularity result Lemma~\ref{thm:tamanini}, $\hat\Omega_n^i\cap\mathcal{O}\to E^i\subset B_{R/2}(0)$ in Hausdorff norm, so in particular
$\hat\Omega^i_n\cap B_{2R}(0)\subset B_R(0)$ for all sufficiently large $n$.  When $i=0$ there is the slightly delicate issue that $\Omega_n$ are not necessarily $(\omega,r)$-minimizers in a neighborhood of $0$.  For $i=0$, define the open set $\tilde{\mathcal{O}}:=B_{2R}(0)\setminus \overline{B_\delta(0)}$, with any $\delta\in (0,R/2)$, so $\Omega_n$ are uniformly $(\omega,r)$-minimizers in $\tilde{\mathcal{O}}$.  Again, by part (ii) of Lemma~\ref{thm:tamanini} we conclude that $\Omega_n\cap [B_{2R}(0)\setminus B_R(0)]=\emptyset$ for all sufficiently large $n$.  

Finally, suppose $\Om_n^N\neq\emptyset$ for all $n\in\N$.  Recall that by Lemma~\ref{lem:CC}, $|\Om_n^N|\to 0$, so $\Om_n^N\to\emptyset$ globally.  As $\Om_n$ is an $(\omega, r)$-minimizing sequence each $\partial^*\Om_n^N$ is a smooth hypersurface, and there would then exist $y_n\in\pt\Om^N_n$ for each $n$.  The translates $\hat\Om^N_n:=\Om^N_n-y_n$ are again smooth, with 
$0\in\partial\hat\Om^N_n$ for each $n$.  Invoking (ii) of Lemma~\ref{thm:tamanini} we arrive at a contradiction, because then $0$ lies on the boundary of the limit set of the $\hat\Om_n^N$, which is empty.  Therefore we must have $\Om_n^N=\emptyset$ for large $n$.

As $|x^i_n-x^j_n|\to\infty$ for $i\neq j$, and each $G^i_n \cap B_R(0)=\emptyset$, the components are well separated, and we obtain \eqref{eq:split3per}.  

We remark that \eqref{eq:split3per} also implies the equality of masses before and after passing to the limit, that is:
\beqn\label{eq:eq}
M=\sum_{i=1}^N m^i = \sum_{i=1}^N m^i_n
\eeqn
holds (with no error) for all $n$ sufficiently large.

\medskip

\noindent {\bf Step 2:} \ $E^0\neq\emptyset$. \  Suppose the contrary.  Since there are only finitely many components, we may choose $k\in \{1,2,\ldots,N\}$ and a subsequence (not relabelled) along which we have
$|x^k_n|=\min\{|x^j_n| \colon j=1,\ldots,N\}$.  Consider the sets $\breve\Omega_n:=\Omega_n-x^k_n$.  The perimeter and nonlocal terms in $\E_{Z}$ are translation invariant, hence, this modification only affects the confinement term $\mathcal{V}$.  By Step~3, we have a disjoint decomposition,
	\[
		  \breve\Omega_n= F^0_n \cup F^k_n \cup \left[\bigcup_{i=1\atop i\neq k}^N (F^i_n+y^i_n)\right] \cup \Omega^N_n,
	\]
where $y^i_n=x^i_n-x^k_n$, with $|y^i_n|\to\infty$, $i\neq k$.  Therefore, 
$\mathcal{V}(F^j_n+x^j_n)\to 0$ and $\mathcal{V}(F^i_n+y^i_n)\to 0$, for all $j=1,\dots,N$ and for all $i\neq k$, while $\mathcal{V}(F^k_n)\to \mathcal{V}(E^k)>0$. Hence,
	\beqns
		\begin{aligned}
		  \E_{Z_n}(\breve \Omega_n) -\E_{Z_n}(\Omega_n) &= -Z_n \calV(F^k_n) - Z_n \sum_{i=1 \atop i\neq k}^N \calV(F^i_n+y^i_n) + Z_n \sum_{i=1}^N \calV(F^i_n+x^i_n) \\
		  																					  & = -Z_n \calV(E^k) + o(Z_n)<0,
		 \end{aligned}
	\eeqns
which contradicts the minimality of $\Omega_n$.  Hence we must have $|E^0|\neq 0$.

\medskip

\noindent {\bf Step 3:} \ {\sl A more refined lower bound.} \  As in Step 1, there exists $R>0$ for which $F_n^i\subset B_R(0)$ for each $n\in\N$ and $i=0,1,\ldots,N$.  Since $\bigcup_{i=0}^N (F_n^i+x_n^i)\subset \Om_n$, we may decompose the nonlocal term and obtain
	\[
		\mathcal{D}(\Omega_n,\Omega_n)\geq \sum_{i,j=0}^N \mathcal{D}(\tilde F_n^i,\tilde F_n^j).
	\]
Let 
	\[
		 R_{n,ij}:=|x_n^i-x_n^j| \qquad \text{ and } \qquad R_{n,i0}:=|x_n^i|.
	\]
	%		 \text{and}\quad R_{n}:=\min\{R_{n,ij},R_{n,i0} \colon i\neq j=0,1,\dots,N\}
  Then, for all $x\in \tilde F_n^i$, $y\in \tilde F_n^j$ and sufficiently large $n$, we have
$$   |x-y| \geq R_{n,ij} - 2R \geq \frac12 R_{n,ij}.  $$    
%	\[
%		 \bigl| |x_n^i-x_n^j|^s - |x-y|^s \bigr| \leq \bigl| (x^i_n-x)-(x^j_n-y)\bigr|^s \leq |x^i_n-x|^s+|x^j_n-y|^s
%   \leq 2R^s
%    \]
%for $0<s\leq 1$.
By the mean value theorem for $f(t)=t^s$ we then calculate,
\begin{align*}
	\bigl| |x_n^i-x_n^j|^s - |x-y|^s \bigr| 
	&\leq s\left(\frac12 R_{n,ij}\right)^{s-1} \left| x_n^i-x_n^j -x+y\right|
	\\
		&\leq  C R_{n,ij}^{s-1}\big(|x_n^i-x|+|x_n^j-y|\big) \leq 2CR\, R_{n,ij}^{s-1}
\end{align*}
Hence, for all sufficiently large $n$,
	\[
	  \left| {1\over |x-y|^s} - {1\over |x^i_n- x^j_n|^s}\right| = {\bigl| |x_n^i-x_n^j|^s - |x-y|^s \bigr|\over |x-y|^s\,|x^i_n-x^j_n|^s}\leq 
     {C\over R_{n,ij}^{s+1}}
    \]
for all $0<s<d$, and we may estimate the off-diagonal terms in the nonlocal energy via
	\begin{multline} \label{eq:offdiag}
			\left|  \mathcal{D}(\tilde F_n^i,\tilde F_n^j) - {m^i_n m^j_n\over |x^i_n-x^j_n|^s}\right|
						 \\ \leq \int_{\tilde F^i_n}\! \int_{\tilde F^j_n} \left| {1\over |x-y|^s} - {1\over |x^i_n-x^j_n|^s}\right| \,dx dy 
						\leq C R_{n,ij}^{-s-1},
	\end{multline}
with a constant $C$ independent of $n$.

The confinement term may be evaluated in a similar way: we have
	\[
	   \left| |x_n^i|^{-p}-|x|^{-p}\right| \leq \sup_{\xi\in \tilde F^i_n} p|\xi|^{-p-1} \, |x-x^i_n| 
\leq C|x^i_n|^{-p-1}\leq CR_{n,i0}^{-p-1},
	\]
and thus 
	\beqn\label{eq:conf}
		  \left| \int_{\tilde F^i_n} \frac{dx}{|x|^p} - {m^i_n\over |x_n^i|^p}\right| \leq CR_{n,i0}^{-p-1}.
	\eeqn
Putting the above estimates together with the perimeter splitting \eqref{eq:split3per}, we obtain a lower bound,
	\begin{align} 
		\E_{Z_n}(\Omega_n)  &\geq \sum_{i=0}^N \E_0 (F^i_n) -Z_n\mathcal{V}(F^0_n)
   +\sum_{i,j=0\atop i\neq j}^N  {m^i_n\, m^j_n\over |x^i_n-x^j_n|^s} \left(1-O(R_{n,ij}^{-1})\right) \nonumber\\
   											&\qquad\qquad\qquad\qquad\qquad\qquad\qquad\qquad\qquad - Z_n\sum_{i=1}^N {m^i_n\over |x^i_n|^p}\left(1+O(R_{n,i0}^{-1})\right) \nonumber \\
   											& \geq \sum_{i=0}^N e_0(m^i_n) -Z_n\mathcal{V}(F^0_n) + 
   \sum_{i,j=0\atop i\neq j}^N  {m^i_n\, m^j_n\over |x^i_n-x^j_n|^s} \left(1-O(R_{n,ij}^{-1})\right) \nonumber\\
   											&\qquad\qquad\qquad\qquad\qquad\qquad\qquad\qquad\qquad - Z_n\sum_{i=1}^N {m^i_n\over |x^i_n|^p}\left(1+O(R_{n,i0}^{-1})\right) \nonumber \\
   \label{eq:lowb}
   											& \geq  \sum_{i=0}^N e_0(m^i_n) -Z_n\mathcal{V}(F^0_n)+ 
   \sum_{i,j=0\atop i\neq j}^N  {m^i\, m^j\over |x^i_n-x^j_n|^s} \left(1-o(1)\right)  \\
   											&\qquad\qquad\qquad\qquad\qquad\qquad\qquad\qquad\qquad - Z_n\sum_{i=1}^N {m^i\over |x^i_n|^p}\left(1+o(1)\right)\nonumber,
\end{align}
where in the last line we have used the convergence $m^i_n\to m^i$.

\medskip

\noindent {\bf Step 4:} \ {\sl A more refined upper bound.} \  In order to obtain a more refined upper bound, let $\Om_t=F_n^0 \cup \left[\bigcup_{i=1}^N (F_n^i+ t\,a^i)\right]$, with sets $F_n^i$ as in Lemma~\ref{lem:CC}, with points $\{a^i\}_{i=1,\dots,N} \subset \Rd$ with $0<|a^i|\leq 1$, and $t>0$ is to be determined optimally.  Substituting $\Om_t$ into $\E_Z$ we recover an upper bound of the same general form as \eqref{upbd} as before, 
\begin{align*}
e_{Z_n}(M)&\leq \E_{Z_n}(\Om_t)  \\
  &\leq \sum_{i=0}^N e_0(m^i_n) - Z_n\mathcal{V}(F^0_n) + \sum_{i,j=0\atop i\neq j}^N 
    \int_{F^i_n+ta^i}\int_{F^j_n+ta^j} {dx\, dy\over |x-y|^s} \\
  &\qquad - Z_n\sum_{i=1}^N \int_{F^i_n+ta^i} |x|^{-p}\, dx
\end{align*}

By the same estimates \eqref{eq:offdiag} and \eqref{eq:conf} as in Step 3 above, we thus have
$$  
\left|  \mathcal{D}(F_n^i,F_n^j) - {m^i m^j\over t^s |a^i-a^j|^s}\right|
						\leq C t^{-s-1}, \qquad
\left| \int_{\tilde F_n^i} \frac{dx}{|x|^p} - {m^i\over t^p|a^i|^p}\right| \leq Ct^{-p-1},
$$
for constant $C$ independent of $t$.  Choosing $t=t_n:=Z_n^{-1/(s-p)}$, we then obtain the upper bound of the form:
\begin{align*}
\ e_{Z_n}(M)\leq \E_{Z_n}(\Om_{t_n}) 
&\leq \sum_{i=0}^N e_0(m^i_n) - Z_n\mathcal{V}(F^0_n)  \\
&\qquad + Z_n^{s/(s-p)}\FNm(0,a^1,\dots,a^N) + O(Z_n^{{s+1\over s-p}}).
\end{align*}
By Proposition~\ref{prop:finitedim}, we may choose $(a^1,\dots,a^N)$ to minimize $\FNm$, and thus obtain the best upper bound,
	\beqn \label{eq:upper}
			\E_{Z_n}(\Om_{t_n}) 
			\leq \sum_{i=0}^N e_0(m^i_n)
			- Z_n\mathcal{V}(F^0_n) + Z_n^{s/(s-p)}\mu_{N,\underline{m}}
			  + o (Z_n^{s/(s-p)}).
	\eeqn

\medskip

\noindent {\bf Step 5:} \ {\sl The scale of $x_n^i=O(Z^{-1/(s-p)})$.} \ 
%First we claim that $R_n\leq C\,Z_n^{-1/(s-p)}$.  Indeed, combining the upper bound \eqref{eq:upper} with the lower bound \eqref{eq:lowb},
%	\begin{equation} \label{eq:finallb}
%	  0> Z_n^{s/(s-p)}\mu_{N,\underline{m}} \geq \E_{Z_n}(u_n)-e_0(M)+Z_n\mathcal{V}(F^0_n) \geq -Z_n R_n^{-p}\, M  [1+o(1)],
%	 \end{equation} 
%from which the desired bound on $R_n$ follows.  
Lastly, we  prove  \eqref{eq:Z20.3}. To this end,  let 
 $ \xi^i_n=x^i_n Z_n^{1/(s-p)}$ for $i=1,\dots,N.$
% we will show that $|\xi_n^i|$ are uniformly bounded for all $i=1,\dots,N$.  
%Suppose not: then there exists $1\leq k\leq N$ such that $|\xi_n^i|\leq C$ holds for all $i=1,\dots,k$, but $|\xi_n^i|\to\infty$ for $i=(k+1),\dots,N$. Hence, substituting this into the 
%
Using  the upper bound \eqref{eq:upper} followed by the 
lower bound \eqref{eq:lowb} we find 
\begin{align*}
  Z_n^{s/(s-p)}\mu_{N,\underline{m}} + o(Z_n^{s/(s-p)})&\geq \E_{Z_n}(\Om_n)
  -\sum_{i=0}^N e_0(m^i_n)+Z_n\mathcal{V}(F^0_n)  \\
     &\geq Z_n^{s/(s-p)} \FNm(0,\xi_n^1,\dots,\xi_n^N)\big( 1+o(1)\big).
\end{align*}
After dividing by $Z_n^{s/(s-p)}$, we conclude that $\{\xi_n^i\}_{i=0,\dots,N}$ is a minimizing sequence for $\FNm$; by Proposition~\ref{prop:finitedim}, the $\xi_n^i$ are in fact bounded, and up to the extraction of a subsequence for each $i=1,\ldots,N$, $\xi_n^i\to y^i$, minimizers of $\FNm$, as $n\to\infty$. We thus obtain \eqref{eq:Z20.3}, and the proof of Theorem~\ref{thm:Z20} is complete. 
\end{proof}

\begin{remark} \label{expansion}
We note that the proof of \eqref{eq:Z20.3} in Step 5 above also shows that we have an expansion of the minimizing energy accurate up to the third-order term, namely,
	\beqn
			%\label{eq:Z20.4}
		\E_{Z_n}(\Om_n) = \sum_{i=0}^N e_0(m^i)  - Z_n\mathcal{V}(F_n^0) + Z_n^{{s\over s-p}} \FNm(0,y_1,\dots,y_N) + o\left(Z_n^{{s\over s-p}}\right), \nonumber
	\eeqn
where $F^0_n$ are the sets constructed in Lemma~\ref{lem:CC}.  One might be tempted to pass to the limit $F_n^i\to E^i$ and express the expansion in terms of the components of the generalized minimizer, but it is not at all clear what the error term in such an expansion would be.
\end{remark}

\bigskip

Finally, we prove Theorem~\ref{thm:BC}.
\begin{proof}[Proof of Theorem~\ref{thm:BC}]  By Step 1 in the proof of Theorem~\ref{thm:Z20} and \ref{thm:tamanini} (see also Theorem~27.5 of \cite{Maggi}) the reduced boundary $\partial^*\Om_n$ is a disjoint union of smooth hypersurfaces.  In fact, by \cite[Theorem 2.7]{BoCr14}, $\partial^*\Om_n$ is of class $C^{3,\beta}$ for $\beta<d-1-s$. In particular, the Euler-Lagrange equation,
\begin{equation}\label{eq:EL1}   (d-1)\kappa(x) + 2v_{\Om_n}(x) - Z_n |x|^{-p} = \lambda_n,  
\end{equation}
is satisfied pointwise on $\partial^*\Om_n$, where $\kappa$ is the mean curvature in $\Rd$, $\lambda_n$ is a Lagrange multiplier, and $v_{\Om_n}(x)$ is the Riesz potential, 
$$  v_{\Om}(x):= \int_\Om  {dy\over |x-y|^s}.  $$
In addition, by the proof of Theorem~\ref{thm:Z20}, $\Om_n$ is $C^{1,\alpha}$ close to the sets
$$   S_n:= \left[ E^0 + \bigcup_{i=1}^N (E^i + x^i_n)\right],  $$
in the sense that for all fixed $R>0$ with $E^i \subset\!\subset B_R(0)$, 
$$
\pt^*\widetilde\Om_n^i:=(\partial^*\Om_n- x^i_n)\cap B_R(0) \to \partial^* E^i \qquad\text{ in $C^{1,\alpha}$},
$$
for all $\alpha\in (0,\frac12)$, and the former are expressed as graphs over the limiting sets $E^i$,
$$   \pt^*  \widetilde\Om_n^i = \{y=\Psi_n(x):= x+\psi_n(x)\nu_i(x): \ x\in \pt^* E^i\},  $$
with $\psi_n(x)\to 0$ in $C^{1,\alpha}$ (see \cite[Theorem 4.2]{AcFuMo13}.)  As each $E^i$ is itself a minimizer of $\E_0$, by the above stated regularity theorem, $\pt^* E^i$ is of class $C^{3,\beta}$, and its normal vector $\nu_{E^i}\in C^2$.  Finally, by \cite[Proposition~2.1]{BoCr14}, the Riesz potentials $v_{\widetilde\Om_n^i}$ are bounded in $C^{1,\beta}(B_R(0))$, so along a subsequence they converge uniformly to $v_{E^i}$ in $B_R(0)$.

For any $\zeta\in C_0^\infty(B_R(0);\R^d)$ we may integrate the Euler-Lagrange equations \eqref{eq:EL1} by parts over $\pt^*\widetilde\Om_n^i$,
\beqn \label{eq:EL}
  \int_{\pt^* \widetilde\Om_n^i} \left( \text{div}_{\tau_n} \zeta - (2v_{\widetilde\Om_n^i}- Z_n|x|^{-p})( \zeta\cdot\nu_n) \right)\, d\Hd  = \lambda_n\int_{\pt^* \widetilde\Om_n^i}  \zeta\cdot\nu_n \, d\Hd,
\eeqn
where $\nu_n:=\nu_{\widetilde\Om_n^i}$ is the normal vector, and the tangential divergence on $\pt^* \widetilde\Om_n^i$ is defined via
$$  \dive_{\tau_n} \zeta = \dive\, \zeta - \nu_n\cdot D\zeta\,\nu_n.  $$
Using the parametrization $\Psi_n$ we can write the above equation with integrals over $\pt^* E^i$, with Jacobian $J_n=|\det D\Psi_n|$.  As $\nu_n\to\nu_{E^i}$, we have $\dive_{\tau_n} \zeta \to \dive_{\tau_{E^i}}\zeta$, and $J_n\to 1$, by the $C^{1,\alpha}$ convergence and $\nu_{E^i}\in C^2$.  Thus, we may pass to the limit $n\to\infty$ in the both integrals in \eqref{eq:EL} and obtain
\begin{multline} \nonumber
 \int_{\pt^* \widetilde\Om_n^i} \left( \dive_{\tau_n} \zeta - (2v_{\widetilde\Om_n^i}- Z_n|x|^{-p})(\zeta\cdot\nu_n) \right)\, d\Hd  \\
 \longrightarrow 
 \int_{\pt^* E^i} \left( \dive_{\tau_{E^i}} \zeta - 2v_{E^i} (\zeta \cdot \nu_{E^i}) \right)\, d\Hd,
\end{multline}
and
\[
\int_{\pt^* \widetilde\Om_n^i}  \zeta\cdot\nu_n \, d\Hd
\longrightarrow
\int_{\pt^* E^i}  \zeta\cdot\nu_{E^i} \, d\Hd.
\]
Thus, $\lambda_n\to \lambda_0$ for some limiting Lagrange multiplier $\lambda_0$.  The values of $\lambda_n$ being (by \eqref{eq:EL1}) the same for each component of $\partial^*\Om_n$, the value of $\lambda_0$ is independent of $i=0,\dots,N$.  Thus, the limiting curvature equation is the same for each limiting set $E^i$, notably with the same Lagrange multiplier $\lambda_0$.  Since for $s<\bar s(d)$ the limiting sets $E^i$ are all balls (cf. \cite[Theorem 2.11]{BoCr14}), and the Lagrange multiplier is uniquely determined by the mass $m^i$ for balls, they must all have the same radius. 
\end{proof}

%%%%%%%%%%%%%%%%%%%%%%%%%%%%%%%%%%%%%%%%%%
%%%%%%%%%%%%%%%%%%%%%%%%%%%%%%%%%%%%%%%%%%%

\bibliographystyle{IEEEtranS}
\def\url#1{}
\bibliography{Z_to_zero}

\end{document}